\documentclass[a4paper,10pt]{amsart}

\usepackage{amsthm}
\usepackage{amssymb}
\usepackage{amsmath}
\usepackage{mathtools}
\usepackage{pdfpages}
\usepackage{changes}
\usepackage{exscale,color,amsopn}
\usepackage[colorlinks=true,pdftex,unicode=true,linktocpage,bookmarksopen,hypertexnames=false]{hyperref}
\usepackage{biblatex}
\usepackage{colortbl}
\usepackage{enumitem}
\usepackage{verbatim}
\usepackage[capitalise]{cleveref}
\usepackage{listings}
\usetikzlibrary{trees}
\renewcommand{\leq}{\leqslant}
\renewcommand{\geq}{\geqslant}

\DeclareMathOperator{\id}{id}

\DeclareMathOperator{\Soc}{Soc}
\DeclareMathOperator{\Sym}{Sym}
\DeclareMathOperator{\Sq}{Sq}

\DeclareMathOperator{\DS}{DS}

\DeclareMathOperator{\supp}{supp}

\newcommand{\C}{\mathbb{C}}
\newcommand{\Z}{\mathbb{Z}}

\newcommand{\Aut}{\operatorname{Aut}}

\newcommand{\Ind}{\mathrm{Ind}}

\newcommand{\Triv}{\operatorname{Triv}}
\newcommand{\G}{\mathcal{G}}

\newcommand{\cox}[1]{\overline{#1}}

\makeatletter
\numberwithin{equation}{section}
\numberwithin{figure}{section}
\numberwithin{table}{section}
\newtheorem{thm}{Theorem}[section]
\newtheorem*{thm*}{Theorem}
\newtheorem{lem}[thm]{Lemma}

\newtheorem{pro}[thm]{Proposition}
\newtheorem{prodef}[thm]{Proposition-Definition}

\newtheorem*{convention*}{Convention}

\theoremstyle{definition}
\newtheorem{defn}[thm]{Definition}

\newtheorem{rem}[thm]{Remark}
\newtheorem{exa}[thm]{Example}

\makeatother

\title{On Dehornoy's representation for the Yang--Baxter equation}
\author{Carsten Dietzel, Edouard Feingesicht, Silvia Properzi}

\address[Carsten Dietzel]{Normandie Univ, UNICAEN, CNRS, LMNO, 14000 Caen, France}
\email{carsten.dietzel@unicaen.fr}

\address[Edouard Feingesicht]{Normandie Univ, UNICAEN, CNRS, LMNO, 14000 Caen, France}
\email{edouard.feingesicht@unicaen.fr}

\address[Silvia Properzi]{Department of Mathematics and Data Science, Vrije Universiteit Brussel, Pleinlaan 2, 1050 Brussel, Belgium}
\email{Silvia.Properzi@vub.be}
	
\date{\today}
\begin{document}

\begin{abstract}
    This article investigates Dehornoy's monomial representations for structure groups and Coxeter-like groups associated with a set-theoretic solution to the Yang--Baxter equation.
    Using the brace structure of these groups and the language of cycle sets, we prove that the irreducibility of the associated monomial representations is equivalent to the indecomposability of the underlying solutions, except when the Dehornoy class is two. For indecomposable solutions, we show that these representations are induced from certain explicitly constructed one-dimensional representations.
\end{abstract}
\maketitle

\section{Introduction}
The \emph{Yang--Baxter equation} (YBE) arises in several areas of mathematics and physics, 
including statistical mechanics,
quantum groups, braid group theory, integrable systems and low-dimensional topology \cite{Turaev_links}. 
Originally introduced in the context of solvable models in statistical mechanics \cite{Baxter_YB,Yang_YB}, 
the YBE can be formulated in a set-theoretic framework following Drinfeld's suggestion \cite{Drinfeld_problems}. 
A \emph{set-theoretic solution to the Yang--Baxter equation} (or simply \emph{solution})
consists of a non-empty set $X$ and a bijection
\[
r:X^2 \longrightarrow X^2, \quad (x,y) \mapsto (\lambda_x(y), \rho_y(x)),
\]
satisfying
\[
(r\times \id)(\id\times r)(r\times \id)=(\id\times r)(r\times \id)(\id\times r).
\]
Solutions are called \emph{involutive} if 
$r^2 =\id_{X^2}$ and \emph{non-degenerate} if the maps $\lambda_x$ and $\rho_x$ are bijections of $X$ for every $x\in X$. 
Besides the simpler, combinatorial nature of this equation, 
Drinfeld’s intention was to subject these solutions to a deformation process
from which new linear solutions can be obtained. 
Since then, set-theoretic solutions to the Yang--Baxter equation 
have attracted significant interest and have been extensively studied,
see for example \cite{CO_simple,Cedo_brace,ESS_group,MR4717099,Rump_DecSqfree}.

A convenient framework to study involutive solutions is provided by \emph{cycle sets}, introduced by Rump \cite{Rump_braces}. 
A \emph{cycle set} is a set $X$ with a binary operation 
$* \colon X^2 \to X$ such that for all $x,y,z \in X$, the identity
\[
(x*y)*(x*z) = (y*x)*(y*z)
\]
holds and all left multiplications
$\sigma_x:y\mapsto x*y$ are bijective. 
Finite cycle sets are in bijective correspondence with finite involutive 
non-degenerate set-theoretical solutions to the Yang--Baxter equation.
More precisely, given a finite cycle set $(X,*)$ then $(X,r_X)$
is an involutive solution,
where 
\[
r_X:X^2\to X^2, \quad
(x,y)\mapsto(\sigma_y^{-1}(x)*y,\sigma_y^{-1}(x)).
\]

In this paper, we focus on involutive 
non-degenerate set-theoretical solutions to the Yang--Baxter equation, 
i.e. solutions whose underlying set is finite.
This allows us to make use of this correspondence 
and to work in the combinatorial setting of cycle sets, 
where the structure is encoded by a single binary operation rather than the maps $\lambda_x$ and $\rho_x$ for $x\in X$. 

A key notion in this context is that of \emph{indecomposability}. 
A solution $(X,r)$  is called \emph{indecomposable}
if there is no partition $X=X_1\sqcup X_2$ with 
$X_i\neq \emptyset$ for $i \in \{1,2\}$
such that $r(X_i^2)=X_i^2$ for $i\in \{1,2\}$.
In terms of cycle sets, this means that the 
associated cycle set 
cannot be partitioned into non-empty subsets that are closed under the operation $*$.
A cycle set with this property is thus also called
\emph{indecomposable}. 
In a sense, indecomposable cycle sets can be viewed as elementary building blocks for general cycle sets. Additionally, they exhibit a much more rigid structure. 
This is reflected in the fact that, 
while the number of cycle sets of a given cardinality grows rapidly, 
the number of indecomposable ones remains comparatively small, 
as shown by computations in \cite{AMV_Cyclesets}
and by the fact that there exists a unique (up to isomorphism) indecomposable cycle set of prime cardinality $p$ \cite[Theorem 2.13]{ESS_group}. 
This indicates strong structural constraints that 
can be effectively studied using algebraic methods. 
Indeed, such an approach was initiated in the works of Etingof, Schedler and Soloviev \cite{ESS_group} and Gateva-Ivanova and Van den Bergh \cite{GIVdB_IType}.
To each involutive solution $(X,r)$, Etingof, Schedler, and Soloviev
associated a \emph{structure group} $G(X,r)$
defined by the generating set $X$ together with the relations
\[
xy = \lambda_x(y) \rho_y(x), \quad x,y \in X.
\]
Equivalently, one can view this group as the structure group of the associated cycle set $(X,*)$, which is defined as
\[
G(X,*)=\langle X\mid (x \ast y) x = (y \ast x) y \rangle.
\]
Both perspectives yield the same group, which encodes combinatorial data of $(X,r)$ 
and provides a natural setting to study properties of the solution.
Moreover, the structure group admits an additional additive structure, 
given by a second operation $+$ which,
together with the multiplicative structure, turns it into a \emph{brace}. 
This construction was introduced by Rump \cite[Section 2]{Rump_braces} 
and later made more explicit by Cedó, Jespers, and Okniński \cite[Theorem 1]{CJO_bracesYBE}. 
A \emph{brace} consists of an abelian group $(B,+)$ endowed with another group structure $(B,\cdot)$ satisfying the compatibility:
\[
a(b + c) = ab - a + ac, \quad a,b,c \in B.
\]
A significant step in understanding the structure group $G(X,*)$
was made by Chouraqui \cite[Theorem 3.3]{Chouraqui_YBEGarside}, 
who showed that it is a Garside group, 
a feature reminiscent of Artin--Tits groups of spherical type. 
Building on this perspective, 
Dehornoy \cite{Dehornoy_RC} introduced a finite quotient $\cox{G}(X,*)$ of $G(X,*)$,
called a \emph{Coxeter-like group},
which plays a role analogous to that of finite Coxeter groups for spherical Artin--Tits groups.  
In addition, Dehornoy introduced \emph{monomial representations}
of both $G(X,*)$ and $\cox{G}(X,*)$ of rank $|X|$.
Under this representation of $G(X,*)$, every element 
is represented by a \emph{monomial} matrix with entries in the fraction field $\C(q)$, 
meaning that this matrix has exactly one nonzero entry in each row and column. 
Such a matrix can be factorized into a permutation part,
determined by the natural action of $G(X,*)$ on $X$ induced by the maps $\sigma_x$,
and a diagonal part, determined by the additive structure of $G(X,*)$.
The monomial representation of $\cox{G}(X,*)$ is then obtained
by specializing $q$ to a primitive $d$-th root of unity,
where $d$ is an integer called the \emph{Dehornoy class of $(X,*)$}.
These representations turn out to be faithful \cite[Proposition 5.13]{Dehornoy_RC}.
As we will see in \cref{sec:prelim}, the latter construction can be extended
to finite quotients $\cox{G}_{l}(X,*)$ of $G(X,*)$ that are
defined for $l\geq 1$ in a way that $\cox{G}_1(X,*)=\cox{G}(X,*)$.
Upon specializing $q$ to a primitive $ld$-th root of unity,
one obtains a faithful monomial representation of  $\cox{G}_{l}(X,*)$.

In this work
we investigate the irreducibility of these monomial representations in terms of the combinatorial structure of the underlying cycle set. 
We show that irreducibility is closely linked to indecomposability. 
Our main result, obtained by combining \cref{thm:irreducibility_of_theta_bar} and \cref{prop:ind_irr},
is the following:
\begin{thm*}
Let $(X,*)$ be a finite cycle set. 
Then the following are equivalent:
\begin{enumerate}
    \item $(X,*)$ is indecomposable.
    \item The monomial representation of the structure group is irreducible.
\end{enumerate}
If moreover $l>1$ or the Dehornoy class of $(X,*)$ is greater than $2$,
then indecomposability is also equivalent to the irreducibility of the monomial representation of $\cox{G}_l(X,*)$.
\end{thm*}
The proof relies on techniques from brace theory and on the interplay between $G(X,*)$ 
and the Sylow subgroups of the \emph{permutation group of $(X,*)$},
\[
\G(X,*)=\langle\sigma_x\mid x\in X\rangle\leq \Sym_{X}.
\]
Indecomposability forces any invariant subspace 
under the monomial representation to contain a basis vector,
which then implies irreducibility. 
Conversely, decomposable solutions naturally give rise to proper invariant subspaces, ensuring reducibility.

This paper is organized as follows.
In Section~\ref{sec:prelim}, we recall the necessary background material on
braces, cycle sets, Coxeter-like groups, Dehornoy’s monomial representations and their generalizations. 
In Section~\ref{sec:irreducibility}, we prove our main result, 
establishing irreducibility criteria in terms of 
indecomposability and the Dehornoy class.
In Section~\ref{sec:induction} we show that,
for indecomposable cycle sets, 
the monomial representations are induced 
from characters of certain subgroups,
via an explicit construction
described in \cref{thm:generic_monomials_are_induced}.

\section{Preliminaries}\label{sec:prelim}

\subsection{Braces}

Rump introduced braces in \cite{Rump_braces} as a generalization of radical rings in order to study involutive non-degenerate solutions of the Yang--Baxter equation.
A \emph{brace}, as reformulated in \cite{CJO_bracesYBE}, is a triple $(B,+,\circ)$, where $(B,+)$ is an abelian group and $(B,\circ)$ is a group such that
\[
a\circ(b+c)=a\circ b-a + a\circ c,
\]
for all $a,b,c\in B$, where $-a$ denotes the inverse of $a$ in $(B,+)$.
More generally, for $n\in\Z$ and $a\in B$, we will denote by $na$ the $n$-th power of $a$ in $(B,+)$ and by $a^n$ the $n$-th power of $a$ in $(B,\circ)$.
Moreover, it follows immediately from the definition that
the neutral elements of the two operations coincide;
hence, we will use a single symbol,
0, to denote both.

If $B$ is a brace, then the \emph{multiplicative group} $(B,\circ)$ acts by automorphisms on the \emph{additive group} $(B,+)$ via the $\lambda$-action, defined as
\[
\lambda_a(b)=-a+a\circ b.
\]

\begin{exa}
    If $(G,+)$ is an abelian group, then $\Triv(G)=(G,+,+)$ is a brace, called the \emph{trivial brace} on $G$. It is easily seen that trivial braces are characterized by the property that $(B,\circ)$ acts trivially on $(B,+)$ by means of the $\lambda$-action.
\end{exa}

A \emph{subbrace} of a brace $B$ is a subset $A \subseteq B$ that is a subgroup of both $(B,+)$ and $(B,\circ)$. A \emph{left ideal} of $B$ is a subset $I$ of $B$ such that $I$ is a subgroup of $(B,+),$ and $\lambda_a(I)\subseteq I$ for all $a\in B$. If additionally, $I$ is a normal subgroup of $(B,\circ)$, one calls $I$ an \emph{ideal}. For example, the \emph{socle} of $B$,
\[
\Soc(B)=\ker\lambda=\{a\in B\mid a\circ b=a+b\text{ for all }b\in B\},
\]
is an ideal of $B$.

\begin{convention*}
    Throughout the rest of this article, we will suppress $\circ$ for the multiplication in a brace and indicate it by juxtaposition.
\end{convention*}

\subsection{Cycle sets}
A \emph{cycle set} is a pair $X=(X,*)$,
where $X$ is a non-empty set
and $*: X \times X \to X$; $(x,y) \mapsto x\ast y$  is a binary operation such that the left multiplication by $x$,
\[
\sigma_x:y\mapsto x\ast y,
\]
is a bijection for each $x\in X$ and
\[
(x\ast y)\ast (x\ast z)=(y\ast x)\ast (y\ast z)
\]
holds for all $x,y,z\in X$. If additionally, the \emph{square map} $\Sq: x\mapsto x\ast x$ is bijective, we say that the cycle set is \emph{non-degenerate}.
Recall that a finite cycle set is always non-degenerate, as proven in \cite[Theorem 2]{Rump_DecSqfree}.

\begin{convention*}
    This article considers non-degenerate cycle sets only, so we will refer to \emph{non-degenerate cycle sets} more briefly as \emph{cycle sets}.
\end{convention*}

A cycle set $(X,\ast)$ is called \emph{indecomposable} if there are no proper partitions $X=X_1\sqcup X_2$ such that $X_1$ and $X_2$
are closed under the cycle set operation.

It can be shown that any brace $(B,+,\circ)$ becomes a cycle set under the mappings $\sigma_b=\lambda_b^{-1}$ for $b\in B$. On the other hand, cycle sets give rise to several braces, as well: for any cycle set $(X,\ast)$, define its \emph{permutation group} as the subgroup
\[
\G(X)=\langle \sigma_x\mid x\in X\rangle \leq \Sym_X,
\]
the latter denoting the symmetric group of permutations on the set $X$, acting from the \emph{left}.
Recall that, as proven in \cite[Proposition 2.12]{ESS_group}, $X$ is indecomposable if and only if $\G(X)$ acts transitively on $X$.
Moreover, for a cycle set $(X,*)$, letting $\lambda_x=\sigma_x^{-1}$, there is a unique way to equip $(\G(X),\circ)$ with an abelian group operation $+$ that satisfies $\lambda_x + \lambda_y = \lambda_x \lambda_{\lambda_x^{-1}(y)}$, such that $(\G(X),+,\circ)$ is a brace.

In the brace structure on $\G(X)$, the $\lambda$-action satisfies the relation
\begin{equation} \label{eq:lambda_compatibility}
\lambda_{\lambda_x}(\lambda_y) = \lambda_{\lambda_x(y)},
\end{equation}
for $x,y \in X$, therefore the $\lambda$-maps of the cycle set $X$ are compatible with the $\lambda$-action of $\G(X)$ under the map $X \to \G(X)$; $x \mapsto \lambda_x$.

More generally, using $\lambda$-notation to denote the image of $x \in X$ under the action of an element $g \in \G(X)$ by $\lambda_g(x)$, \cref{eq:lambda_compatibility} extends to 
\begin{equation} \label{eq:lambda_compatibility_2}
    \lambda_g(\lambda_x) = \lambda_{\lambda_g(x)}
\end{equation}
for $g \in \G(X)$, $x \in X$. As a consequence, the $\lambda$-action on the brace $\G(X)$ is compatible with the $\lambda$-action on $X$.

\begin{convention*}
    Due to \cref{eq:lambda_compatibility_2}, we use $\lambda$-notation both in the case when a brace derived from a cycle set $X$ (like $\G(X), G(X), \ldots$) acts on $X$ and in the case when such a brace acts on itself via its $\lambda$-action. No danger of confusion will arise from this convention as it will always be clear which set is acted upon.
\end{convention*}

\begin{exa} \label{exa:cyclic_cycle_set}
    For some $n>0$, let $X = \{x_1,\ldots,x_n\}$ and fix the permutation cycle $\sigma = (1 \ 2 \ldots \ n)$. Then the operation $x_i \ast x_j = x_{\sigma(j)}$ turns $X$ into a cycle set, called a \emph{cyclic} cycle set. It can be shown that $(X, \ast)$ is an indecomposable cycle set with permutation group $\G(X) = \left\langle (1 \ 2 \ldots n) \right\rangle \cong \Z_n$, the cyclic group of order $n$. Furthermore, $\G(X)$ is a trivial brace when equipped with the canonical brace structure described above.
\end{exa}

With any cycle set $(X,\ast)$ we also associate its \emph{structure group}
\[
G(X)=\langle X\mid (x \ast y) x = (y \ast x) y \rangle
\]
that contains $X$, as the canonical map $X\to G(X)$, $x\mapsto x$ can be shown to be injective \cite{ESS_group}.

Moreover, there is a unique way of defining an addition on $G(X)$ that extends 
$$x+y=x  \sigma_x(y) = x \lambda_x^{-1}(y)$$
for $x,y\in X$ and that provides $G(X)$ with a brace structure (see \cite{Cedo_brace}, for instance).
Under the cycle set structure of the thus constructed brace $G(X)$, the embedding $X \hookrightarrow G(X)$ identifies $X$ with a sub-cycle set of $G(X)$.

Note also that there is a surjective brace homomorphism $G(X)\to \G(X)$ given by $x\mapsto \lambda_x=\sigma^{-1}_x$ which has kernel $\Soc(G(X))$. This implies also that the map $\lambda: X\to \G(X)$, $x\mapsto \lambda_x$ is a cycle set homomorphism.

An important invariant of a non-degenerate cycle set $X$ that we are going to use throughout the whole paper is the  \emph{Dehornoy class}. It is the smallest positive integer $d$ such that $dx \in \Soc(G(X))$ for every $x\in X$.

There is also a different interpretation of the Dehornoy class, given in \cite{MR4717099}:
\begin{pro}
    The Dehornoy class $d$ of a finite cycle set $X$ is the least common multiple of the additive orders of the generators $\lambda_x \in \G(X)$ for $x \in X$. Equivalently, $d$ is the exponent of the group $(\G(X),+)$. 
\end{pro}

Let $\pi(k)$ denote the set of prime divisors of an integer $k$. In \cite{Feingesicht_class}, the following properties are proven:
\begin{pro}
\label{primes dividing d and n}
    Let $X$ be a cycle set of size $n$ and Dehornoy class $d$.
     Then $d$ divides $|\G(X)|$ and $|\G(X)|$ divides $d^n$, so $\pi(d)=\pi(|\G(X)|)$.
     
     In particular, if $X$ is indecomposable, then 
     $\pi(n)\subseteq \pi(d)=\pi(|\G(X)|)$
\end{pro}

Let $X$ be a cycle set with Dehornoy class $d$, then for any integer $l \geq 1$, the additive subgroup $ldG(X) = \{ (ld)g : g \in  G(X) \} \leq (G(X),+)$ is a left ideal of $G(X)$ that is contained in $\Soc(G(X))$ and therefore it is an ideal of $G(X)$. 
Consequently, we can form the quotient, which is then a skew brace.
\begin{defn} \label{defn:coxeter_like_group}
    Let $X$ be a cycle set with Dehornoy class $d$. Given an integer $l \geq 1$, we define the \emph{Coxeter-like group}
    \[
    \cox{G}_l(X) = G(X)/ldG(X).
    \]
\end{defn}

Note that the additive group of $\cox{G}_l(X)$ is isomorphic to $\Z_{ld}^X$. Also observe that, as $ldG(X) \subset\Soc(G(X))$, the canonical morphism $G(X)\to\mathcal G(X)$ factors through $\cox{G}_l(X)$, i.e $\mathcal G(X)=\cox{G}_l(X)/\Soc(\cox{G}_l(X))$.

The importance of the Coxeter-like groups lies in the fact that they play the role of a germ of the Garside structure on $G(X)$ (see \cite{Dehornoy_RC}).

When $l=1$, we will simply write $\overline G_1(X)$ as $\overline G(X)$.

\begin{convention*}
    From now on, we will often abbreviate $G(X)$ with $G$ when the context is clear. The same convention applies to $\G(X)$, $\cox{G}(X)$ and $\cox{G}_l(X)$.
\end{convention*}

\subsection{Induced representations}
In what follows, all modules are left modules.
Let $G$ be a group, $R$ a commutative ring and $V$ an $R$-module. 
It is well-known that a representation $\rho\colon G\to \Aut(V)$ is equivalent to an $R[G]$-module structure on $V$.

\begin{prodef}[{\cite[§43]{CR_Repr}}]
\label{pdef:induced}\index{Induced representation}
If $H$ is a subgroup of $G$ and $V$ is an $R[H]$-module, the \emph{induced $R[G]$-module} is defined as $\Ind_H^GV = R[G]\otimes_{R[H]} V$.
\end{prodef}

By the correspondence between representations and modules, the \emph{induced representation} $\Ind_H^G\rho$ is the representation of $G$ associated with the $R[G]$-module $\Ind_H^GV$,
where $V$ is the $R[H]$-module associated with the representation $\rho$.

It is well-known that induced representations are connected to systems of imprimitivity.

\begin{defn}[{\cite[§50.1]{CR_Repr}}] \label{defn:system_of_imprimitivity}
    Let $V$ be an $R[G]$-module. A family of $R$-submodules $(U_i)_{i \in I}$ of $V$ is a \emph{system of imprimitivity} if the following three conditions are satisfied:
    \begin{enumerate}
        \item $V = \bigoplus_{i \in I} U_i$,
        \item $G$ permutes the family, i.e. for all $i \in I$, $g \in G$, there is a $j \in I$ such that $g\cdot U_i = U_j$,
        \item $G$ acts transitively on the family, i.e. for all $i,j \in I$, there is a $g \in G$ such that $g\cdot U_i = U_j$.
    \end{enumerate}
\end{defn}

\begin{pro} \label{pro:imprimitive_modules_are_induced}
    Let $V$ be an $R[G]$-module and $(U_i)_{i \in I}$ a system of imprimitivity thereof. For an $i_0 \in I$, let $G_0 = \{ g \in G: g\cdot U_{i_0} = U_{i_0} \}$. Then, by restriction, $U_{i_0}$ is an $R[G_0]$-module and there is a canonical isomorphism of $R[G]$-modules
    \[
    V \cong \Ind_{G_0}^G U_{i_0}.
    \]
\end{pro}

\begin{proof}
    \cite[§50.2]{CR_Repr}.
\end{proof}

\begin{rem}
    Although \cite{CR_Repr} treats the concept of (systems of) imprimitivity in terms of modules over group rings over a \emph{field}, \cref{defn:system_of_imprimitivity} and \cref{pro:imprimitive_modules_are_induced} are also valid for group rings over a commutative ring.

    Moreover, we decided to change the notation for induced modules in the reference in favour of the more suggestive notation $\Ind_H^G$.
\end{rem}

\subsection{Monomial representation}
In 2015, Dehornoy \cite{Dehornoy_RC} developed a calculus of words to study structure groups of cycle sets and their Garside structure. In particular, given a cycle set $X$, he deduced the existence of a monomial representation of its structure group
$\Theta: G(X)\to M_X(\C(q))$, where $M_X(\C(q))$ is the ring of matrices with entries in $\C(q)$ indexed by $X\times X$. Moreover, he shows that this representation descends to a representation of the Coxeter-like group $\cox{G}(X)$ when specializing $q$ to a $d$-th root of unity $\zeta_{d}$, where $d$ is the Dehornoy class of $X$.

Given a permutation $\sigma \in \Sym_X$, we write $P_{\sigma} = (P_{ij})_{i,j \in X}$ for its \emph{permutation matrix} whose entries are given by
\[
P_{ij} = \begin{cases}
    1 & i = \sigma(j) \\
    0 & i \neq \sigma(j).
\end{cases}
\]
This matrix acts on a basis vector $e_i$ with $i \in X$ as $P_{\sigma}(e_i) = e_{\sigma(i)}$. In $M_X(\C(q))$ for $x \in X$, denote by $D_x$ the diagonal matrix
$\text{diag}(1,\dots,1,q,1,\dots,1)$ with a $q$ on the $x$-coordinate. With this notation, we define for $x \in X$ the permutation matrix $P_x = P_{\lambda_x}$ and, more generally, $P_g = P_{\lambda_g}$ for $g \in \G(X)$. Moreover, expressing $g\in G(X)$ in the brace structure as $g=\sum\limits_{x\in X}g_x x$, 
for some $g_x\in \in \Z$,
we define the diagonal matrix $D_g=\prod_{x\in X} D_x^{g_x}$.

\begin{thm}[\cite{Dehornoy_RC}, Proposition 5.13] \label{thm:dehornoy_monomial_representation}
Let $X$ be a cycle set of Dehornoy class $d$.
The map $X\to M_X(\C(q))$ defined as $x\mapsto D_xP_x$ extends to a faithful representation 
$$\Theta\colon G(X)\to M_X(\C(q))$$
and the specialization at $q= \zeta_{d}$ yields a faithful representation 
\begin{align*}
    \cox{\Theta} \colon \cox G(X)=G(X)/dG(X) & \to M_X(\C) \\
    g \cdot dG(X) & \mapsto \Theta(g)_{q = \zeta_{d}}.
\end{align*}
Furthermore, the matrix $\Theta(g)$ can be uniquely written as the product of a diagonal matrix and a permutation matrix:
\begin{equation} \label{eq:theta_in_terms_of_g}
\Theta(g) =D_gP_g.
\end{equation}
\end{thm}

\begin{defn}
    The representations $\Theta$ and $\cox{\Theta}$ are called the \emph{monomial representations} of $G(X)$ and $\cox G(X)$, respectively.
\end{defn}

We now extend the above construction to arbitrary roots of unity. Although this generalization is not explicitly stated in \cite{Dehornoy_RC}, it follows by adapting the same arguments.

\begin{pro} \label{prop:monomial_representation_general_l}
Let $d$ be the Dehornoy class of $X$. Then, for any integer $l \geq 1$, specializing at $q=\zeta_{ld}=\exp\left(\frac{2\pi i}{ld}\right)$ yields a faithful representation
\begin{align*}
    \cox{\Theta}_l \colon \cox G_l(X)=G(X)/ldG(X) & \to M_X(\C), \\
    g \cdot ldG(X) & \mapsto \Theta(g)_{q = \zeta_{ld}}.
\end{align*}
\end{pro}

\begin{proof}
We follow the argument of \cite[Proposition 5.13]{Dehornoy_RC}, now for general $l$.
Every matrix $\Theta(g)$ decomposes uniquely as
\[
\Theta(g) = D_g P_g,
\]
where $D_g$ is diagonal and $P_g$ is a permutation matrix.  
Writing $g = \sum_{x \in X} g_x x$, the $x$-th diagonal entry of $D_g$ is $q^{g_x}$,
which specializes to $\zeta_{ld}^{\, g_x}$ at $q = \zeta_{ld}$.  
Hence an element $g = \sum_{x \in X} g_x x\in G(X)$ lies in the kernel of the map
\[
\theta: G(X)\to  M_X(\C),\quad g \mapsto \Theta(g)_{q=\zeta_{ld}}
\]
if and only if $D_g=I$ and $P_g=I$ after specialization at $q=\zeta_{ld}$.

Now $D_g=I$ at $q=\zeta_{ld}$ if and only if
all coefficients $g_x$ are divisible by $ld$, that is, $g\in ldG(X)$.
Moreover $P_g=I$  at $q=\zeta_{ld}$ if and only if $g\in \Soc(G(X))$.

Therefore $g\in \ker \theta$ if and only if $g\in ldG(X)\cap \Soc(G(X))=ldG(X)$,
Since, by definition of the Dehornoy class, $ldG(X)\subseteq \Soc(G(X))$, 
it follows that $\ker\theta$ is exactly $ldG(X)$.
Consequently, the induced map 
$\cox{\Theta}_l\colon \cox G_l(X) \to M_X(\C)$ 
is well-defined and injective.
\end{proof}

\section{Irreducibility} \label{sec:irreducibility}
Let $X$ be a cycle set of size $n$ and Dehornoy class $d$, with permutation group $\G$ and structure group $G$.  Recall that $\sigma_x = \lambda_x^{-1}$, and that $|\G|$ divides $d^n$ (\cref{primes dividing d and n}).

As $G$ is a brace and $(G,+)$ is generated by $X$,
we can express any $g$ in $G$ as $g=\sum_{x\in X} g_x x$ with $g_x \in \Z$ for $x \in X$.

\begin{pro}\label{pro:irr_implies_ind} 
If $X$ is a decomposable cycle set, the representations $\Theta$ and $\overline\Theta_l$ are reducible.
\end{pro}
\begin{proof}
If $X$ is decomposable, say $X=X_1\sqcup X_2$, then the action of $G$ stabilizes the proper subspaces spanned by $X_1$ and $X_2$ in $\C(q)^X$, thus $\Theta$ is reducible. The proof for $\cox{\Theta}_l$ is exactly the same.
\end{proof}

In the following, we will prove that $\cox{\Theta}$ (i.e. $\cox{\Theta}_1$) is irreducible for most indecomposable cycle sets $X$.

\begin{thm} \label{thm:irreducibility_of_theta_bar}
    Let $X$ be an indecomposable cycle set of size $n$ and Dehornoy class $d$. $\cox{\Theta}$ is irreducible if one of the following conditions is satisfied:
    \begin{enumerate}
        \item $d > 2$,
        \item $d = 2$ and $|\G(X)| < 2^{\frac{n}{2}}$.
    \end{enumerate}
\end{thm}

We provide some machinery first.

Let $p$ be a prime and $n > 0$. We can always uniquely factorize $n = p^v m$ with $v, m \geq 0$ and $p \nmid m$. Therefore, one can define the \emph{$p$-valuation} $v_p(n)$ as the exponent $v$ in such a factorization.

On the other hand, there is a unique $p$-adic representation $n = \sum_{i = 0}^{\infty} a_ip^i$ with $0 \leq a_i < p$ for all $i \geq 0$. The \emph{$p$-adic digit sum} of $n$ is defined as $\DS_p(n) = \sum_{i=0}^{\infty} a_i$.

We will need the following elementary result about the $p$-valuation of factorials:

\begin{lem} \label{lem:p_valuation_of_factorials}
    For all $n \geq 0$, we have $v_p(n!) = \frac{n - \DS_p(n)}{p-1}$.
\end{lem}

\begin{proof}
\cite[Lemma 4.2.8.]{Cohen_NT_V1}.
\end{proof}

The following estimate for the $p$-valuation of $|\G(X)|$ is now immediate:

\begin{lem} \label{lem:estimate_size_of_sylows}
    Let $X$ be a cycle set of size $n$. Then, $v_p(|\G(X)|) \leq \frac{n-1}{p-1}$ for any prime $p$.
\end{lem}

\begin{proof}
    As $\G(X)$ is a subgroup of $\Sym_X$, the order $|\G(X)|$ divides $|\Sym_X| = n!$ and therefore,
    \[
    v_p(|\G(X)|) \leq v_p(n!) = \frac{n-\mathrm{DS}_p(n)}{p-1} \leq \frac{n-1}{p-1}.
    \qedhere\]
\end{proof}

Denote the set of invertible diagonal $n \times n$-matrices over a field $K$ by $\mathcal{D}_n(K)$. We will use the following general lemma:

\begin{lem} \label{lem:diagonal_actions}
    Let $n > 0$, $G$ be a group and let $\rho: G \to \mathcal{D}_n(K)$ be a representation. Let $\rho(g) =\mathrm{diag}(d_{1,g},\ldots,d_{n,g})$ and suppose that for any $1 \leq i < j \leq n$, there is a $g \in G$ such that $d_{i,g} \neq d_{j,g}$. Then every $G$-invariant subspace $0 \neq U \leq K^n$ contains some unit vector $e_i$.
\end{lem}

\begin{proof}
    Let $0 \neq U \subseteq K^n$ be $G$-invariant and pick $0 \neq v \in U$ whose \emph{support} $\supp(v) = \{ i \in \{1,\ldots,n \} : v_i \neq 0 \}$ is as small as possible. If $|\supp(v)| = 1$, then $v = a e_i$ for some $0 \neq a \in K$ and some index $i$, and the claim is proven. Suppose that $|\supp(v)| > 1$ and choose indices $i < j$ with $i,j \in \supp(v)$. By assumption, there is a $g \in G$ such that $d_{i,g} \neq d_{j,g}$. Consider the vector $w = \rho(g)(v) - d_{j,g}v \in U$. As $\rho$ acts by diagonal matrices, it is immediate that $\supp(w) \subseteq \supp(v)$.
    For this vector, we observe:
    \begin{align*}
    w_i & = d_{i,g}v_i - d_{j,g}v_i = (d_{i,g} -d_{j,g})v_i \neq 0 ; \\ 
    w_j & = d_{j,g}v_j - d_{j,g}v_j = 0.
    \end{align*}
    The first calculation allows us to conclude that $w \neq 0$. The other calculation shows that $j \not\in \supp(w)$ which implies $|\supp(w)| \leq |\supp(v)| -1$. But this contradicts the assumption that the support of $v$ is as small as possible among the nonzero vectors in $U$.
\end{proof}

\begin{lem}\label{ei} Let $X$ be an indecomposable cycle set. Furthermore, let $U$ be an invariant subspace of $\C^X$ under $\cox{\Theta}_l$ (resp. an invariant subspace of $\C(q)^X$ under $\Theta$). If $e_x$ is in $U$ for some $x \in X$, then $U =\mathbb C^X$ (resp. $U = \C(q)^X$).
\end{lem}
\begin{proof}
We will only consider the representation $\Theta$ as the proof is similar for $\cox{\Theta}_l$. By indecomposability, for all $e_y$, there exists $g\in G$ such that $P_g e_x=e_y$, therefore an application of \cref{eq:theta_in_terms_of_g} shows that $\Theta(g)e_x=D_g P_g e_x=D_g e_y \in U$. As $D_g$ is diagonal, $\Theta(g)e_x$ is a scalar multiple of $e_y$, so $e_y\in U$. It follows that $U = \C(q)^X$.
\end{proof}

We can now proceed with the proof of the main theorem of this section:
\begin{proof}[Proof of \cref{thm:irreducibility_of_theta_bar}]
    Let $p$ be a prime dividing $d$ and write $d = m p^v$ with $p \nmid m$.
    
    Write $\cox{G}^{(p)}$ for the (additive) $p$-Sylow subgroup of $\cox{G}$ and consider the $p$-Sylow subgroup of the socle, $S^{(p)} = \Soc(\cox{G})^{(p)} \subseteq \cox{G}^{(p)}$. Then $\cox{\Theta}$ restricts to a faithful diagonal representation $\cox{\Theta}|_{S^{(p)}}: S^{(p)} \to \mathcal{D}_X(\C)$, such that the diagonal matrices in the image have $p^v$-th roots of unity on the diagonal.

    As $(\cox{G},+) = \Z_d^X$, it follows that $\cox{G}^{(p)} = m\cox{G}$, so $(\cox{G}^{(p)},+) \cong \Z_{p^v}^X$ and we can express each element $g \in \cox{G}^{(p)}$ uniquely as
    \[
    g = \sum_{\overline{x} \in \overline{X}} g_{m\overline{x}} m\overline{x},
    \]
    with $ \ 0 \leq g_{m\overline{x}} < p^v$
    and where $\overline{x}$ is the class of $x$ in $\cox{G}$, for all $x \in X$
    and $\overline{X}=\{\overline{x}\mid x\in X\}$.
    Observe that for $h \in \cox{G}$ and $g \in \cox{G}^{(p)}$, we have
    \[
    \lambda_h(g) = \sum_{\overline{x} \in \overline{X}} g_{m\overline{x}} m \lambda_h(\overline{x}) = \sum_{x \in X} g_{m\lambda_h^{-1}(\overline{x})} m\overline{x},
    \]
    which shows
    \begin{equation} \label{eq:equivariance_of_coordinates}
        (\lambda_h(g))_{m\overline{x}} = g_{m\lambda_h^{-1}(\overline{x})}
    \end{equation}
    
    Using this notation for $s\in S^{(p)}$, we define on $X$ the equivalence relation:
    \[
    x \sim y \Leftrightarrow \forall s \in S^{(p)}: s_{m\overline{x}} = s_{m\overline{y}}.
    \]
    If $x \sim y$ and $g \in \cox{G}$, then for all $s \in S^{(p)}$, we also have
    \[
    s_{m\lambda_g(\overline{x})} = (\lambda_g^{-1}(s))_{m\overline{x}} = (\lambda_g^{-1}(s))_{m\overline{y}} = s_{m\lambda_g(\overline{y})}.
    \]
   
    Here we use \cref{eq:equivariance_of_coordinates} and the fact that $S^{(p)}$ is a left ideal in $\cox{G}$: as $\Soc(\cox{G})$ is an ideal in $\cox{G}$ and $S^{(p)}$, its unique $p$-Sylow subgroup, is characteristic therein, each $\lambda_g$ for $g \in \cox{G}$ has to restrict to an automorphism of $\Soc(\cox{G})$ that leaves $S^{(p)}$ invariant.
    
    Therefore, the equivalence relation $\sim$ is $\cox{G}$-invariant and the classes in $X/\!\!\sim$ are blocks of imprimitivity for the action of $\cox{G}$ on $X$. In particular, all blocks have the same size. Write $\mathcal{B}_x = \{y \in X: x \sim y \}$ for $x \in X$.

    If $\sim$ is not a trivial equivalence relation, then $|\mathcal{B}_x| \geq 2$ for all $x \in X$ which shows that the number of blocks is $|X/\!\!\sim \!\!| \leq \frac{n}{2}$. As by definition, the coordinates $s_{m\overline{x}}$ are blockwise constant for $s \in S^{(p)}$, and can take $p^v$ different values, we conclude that $|S^{(p)}| \leq (p^v)^{\frac{n}{2}}$ and thus, $v_p(|S^{(p)}|) \leq \frac{vn}{2}$. This implies that, if $\sim$ is nontrivial,
    \[
    v_p(|\G^{(p)}|) = v_p \left( \frac{|\cox{G}^{(p)}|}{|S^{(p)}|} \right) = vn - \frac{vn}{2} \geq \frac{vn}{2}.
    \]
    If $p > 2$, then, by \cref{lem:estimate_size_of_sylows}, we can estimate 
    \[
    v_p(|\G^{(p)}|) \leq \frac{n-1}{p-1} < \frac{n}{2} \leq \frac{vn}{2}.
    \]
    In this case, $\sim$ is a trivial equivalence relation. If $p = 2$ and $v > 1$, we use the same lemma to conclude that
    \[
    v_2(|\G^{(2)}|) \leq n-1 < \frac{vn}{2}
    \]
    which again implies that $\sim$ is trivial. We conclude that if $d > 2$, we can always find a prime $p|d$ for which the equivalence relation $\sim$ is trivial.
    
    Finally, let $d=2$, i.e. $p=2$ and $v=1$, and suppose that $|\G| < 2^{\frac{n}{2}}$, then also $|\G^{(2)}| < 2^{\frac{n}{2}}$, therefore
    \[
    v_2(|\G^{(2)}|) < \frac{n}{2} = \frac{vn}{2},
    \]
    and $\sim$ is trivial.

    We have therefore proven that in both cases considered in the theorem, we can find a prime $p|d$ such that the associated equivalence relation $\sim$ is trivial on $X$.
    
    For $s \in S^{(p)}$, \cref{eq:theta_in_terms_of_g} implies that $\cox{\Theta}(s)$ is a diagonal matrix with entries $d_{x,s} = \zeta_d^{ms_{m\overline{x}}} = \zeta_{p^v}^{s_{m\overline{x}}}$, therefore triviality of $\sim$ means that for any $x,y \in X$ with $x \neq y$, there is an $s \in S^{(p)}$ such that $d_{x,s} \neq d_{y,s}$. 

    Let now $0 \neq U \subseteq \C^X$ be a $\cox{G}$-invariant subspace with respect to the action of $\cox{\Theta}$. As $U$ is, in particular, $S^{(p)}$-invariant, \cref{lem:diagonal_actions} now tells us that $e_x \in U$ for some $x \in X$. But by \cref{ei}, it follows that $U = \C^X$. As $U$ was arbitrary, it follows that $\cox{\Theta}$ is irreducible in the considered cases.
\end{proof}

It turns out that $l=1$ is the only case where indecomposability of $X$ does not always guarantee that $\cox{\Theta}_l$ is irreducible. Using way simpler techniques, we can prove:

\begin{pro}\label{prop:ind_irr}
Let $l$ be a positive integer. Then the following are equivalent:
\begin{enumerate}[label=(\roman*)]
\item $X$ is indecomposable
\item $\Theta\colon G\to M_X(\C(q))$ is irreducible
\end{enumerate}
Moreover, if $l>1$, these conditions are also equivalent to
\begin{enumerate}[label=\alph*)]
\item[(iii)] $\overline\Theta_l\colon\overline G_l\to M_X(\C)$ is irreducible
\end{enumerate}
\end{pro}
\begin{proof}
\textit{(ii) $\Rightarrow$ (i)} and $\textit{(iii) $\Rightarrow$ (i)}$ have already been dealt with in \cref{pro:irr_implies_ind}, so we are left with proving the implications \textit{(i) $\Rightarrow$ (ii)} and, for $l > 1$,  \textit{(i) $\Rightarrow$ (iii)}.

Suppose that $X$ is indecomposable. Let $U$ be a non-trivial subspace of $\C(q)^X$ that is $G$-invariant. As $X$ is of class $d$, we see for all $x \in X$ that $dx\in\text{Soc}(G)$ which implies that $\Theta(dx)= D_x^d = \mathrm{diag}(1,\ldots,q^d,\ldots,1)$ with $q^d \neq 1$ in the position of $x$. Considering all matrices $\Theta(dx)$ ($x \in X$) proves that for any $x,y \in X$ there is a $g \in \Soc(G)$ with $d_{x,g} \neq d_{y,g}$. By \cref{lem:diagonal_actions}, $U$ contains a unit vector $e_x$ and \cref{ei} implies that $U = \C(q)^X$. Therefore, $\Theta$ is irreducible.

If $l > 1$, we see for $x \in X$ that $\cox{\Theta}_l(dx)= \mathrm{diag}(1,\ldots,\zeta_{ld}^d,\ldots,1)$ with $\zeta_{ld}^d = \zeta_l \neq 1$ in the position of $x$. If $X$ is indecomposable, the same line of reasoning can now be applied to prove the irreducibility of $\cox{\Theta}_l$ for $l > 1$.
\end{proof}
\begin{rem}\label{rmk:d2_irr}
For $l=1$,  \cref{pro:irr_implies_ind} still shows that the irreducibility of $\overline \Theta\colon \overline G\to M_X(\C)$ implies the indecomposability of the underlying solution, but the other implication does not hold.  

Indeed, consider \cref{exa:cyclic_cycle_set} for $n=2$: in that case, $X=\{x_1,x_2\}$ and $x_i \ast x_j = x_{\sigma(j)}$ where $\sigma = (1 \ 2)$. Then $X$ has Dehornoy class 2 and $|\G(X)| =2 \geq  2^{\frac{n}{2}}$. 
Furthermore
\[
\overline\Theta(x_1)=\begin{pmatrix}
0&-1\\1&0\end{pmatrix} \quad \textnormal{and} \quad \overline\Theta(x_2)=\begin{pmatrix}
0&1\\-1&0\end{pmatrix}=-\overline\Theta(x_1).
\] These matrices are simultaneously diagonalizable over $\C$, the eigenvalues being $\pm i$, so $\cox{\Theta}$ is not irreducible.
\end{rem}

We close this section with a number-theoretic application of our results.

\begin{exa} For $n>2$, consider the cyclic cycle set $X=\{x_1,\dots,x_n\}$ from \cref{exa:cyclic_cycle_set} with $x_i \ast x_j = x_{\sigma(j)}$ where $\sigma = (1 \ 2  \ldots \ 
n)$. Then $X$ is of class~$n$ and indecomposable. As $d = n > 2$, $\overline\Theta$ is irreducible by \cref{thm:irreducibility_of_theta_bar}. Moreover, by \cite[Theorem 5]{serre_repr}, a representation~$\rho$ of a finite group $G$ is irreducible if and only if $\frac{1}{|G|}\sum\limits_{g\in G} |\text{Tr}(\rho(g))|^2=1$. We now apply this formula to the representation $\overline\Theta$:

For any $g\in\overline G$, write $g=\sum\limits_{x_i\in X} a_i x_i$ ($0 \leq a_i < n$) and define its \emph{length} as $\overline\ell(g)=\sum\limits_{x_i\in X}a_i$.
Then $\lambda_g =\sigma^{-\overline\ell(g)}$. Moreover, $\sigma^k$ stabilizes a point if and only if $k$ is a multiple of $n$, and in this case, $\lambda_g$ stabilizes $X$ point-wisely. Thus the trace of $g$ is non-trivial if and only if $n$ divides $\overline\ell(g)$ and in this case, $\text{Tr}(\overline\Theta(g))=\sum_{i=1}^n \zeta_d^{a_i}$.

As $\overline\Theta$ is irreducible, we have $\frac{1}{|\overline G|}\sum\limits_{g\in\overline G}|\text{Tr}(\overline\Theta(g))|^2=1$. We conclude that $$\sum\limits_{\substack{0\leq a_1,\dots,a_n<n\\a_1+\dots+a_n \equiv 0\text{ mod } n}} \left|\zeta_d^{a_1}+\dots+\zeta_d^{a_n}\right|^2=n^n.$$
\end{exa}

\section{Induction} \label{sec:induction}

The aim of this section is the description of monomial representations in terms of induced representations.

We now choose an element $x_0 \in X$ and define $G_0 = \{ g \in G\mid \lambda_g(x_0) = x_0 \}$ and $\overline G_{l,0}=\{g\in\overline G_l\mid \lambda_g(x_0)=x_0\}$ for $l\geq 1$.
Recall that we can write any element $g$ in the structure brace $G$ as $g=\sum\limits_{x\in X} g_xx$.

\begin{pro}\label{prop:c0}
    The mapping
    \[
    c_0: G \to \Z; \quad g \mapsto g_{x_0}
    \]
    satisfies the following property: for $g \in G_0$, $h \in G$, we have
    \[
    c_0(gh) = c_0(g) + c_0(h).
    \]
    In particular, the restriction $c_0|_{G_0}: G_0 \to \Z$ is a group homomorphism.
\end{pro}

\begin{proof}
   We have $gh=g+\lambda_g(h)=\sum\limits_{x\in X}g_xx+\sum\limits_{x\in X}h_x\lambda_g(x)$.
   If $g\in G_0$, then $\lambda_g(x_0)=x_0$ which implies $(\lambda_g(h))_{x_0} = h_{x_0}$. Thus,
   \[
   c_0(gh)=g_{x_0}+(\lambda_g(h))_{x_0}= g_{x_0} + h_{x_0} = c_0(g)+c_0(h),
   \]
   which proves the first statement. The second statement of the proposition is now immediate.
\end{proof}

By Proposition \ref{prop:c0}, we can define the character $\chi_0: G_0 \to \C[q^{\pm 1}]\subset \C(q)$ by $g \mapsto q^{c_0(g)}=q^{g_{x_0}}$. 
\begin{lem}
The character $\chi_0\colon G_0\to\C[q^{\pm 1}]$ descends to a character $\overline\chi_{l,0}\colon\overline G_{l,0}\to \C$.
\end{lem}
\begin{proof}
With the specialization $\text{ev}_{ld}\colon\C[q^{\pm 1}]\to\C$; $q \mapsto \zeta_{ld}$, we obtain the character $\text{ev}_{ld}\chi_0: G_0 \to \C$; $g \mapsto \zeta_{ld}^{g_{x_0}}$. Recall that $\overline G_l=G/ldG$ and $ldG\subseteq\text{Soc}(G)\subseteq G_0$. Thus, $\text{ev}_{ld}\chi_0$ factors uniquely as $\overline\chi_{l,0}$ through the canonical projection $G_0 \twoheadrightarrow G_0/ldG$. Furthermore,
\[
\text{ker}(G_0\to\overline G_{l,0})=G_0\cap \text{ker}(G\to\overline G_l)=G_0\cap ldG=ldG.
\]
Thus, $G_0/ldG=\overline G_{l,0}$ and $\overline\chi_{l,0}$ is well-defined.
\end{proof}

For a commutative ring $R$, a group $G$ and a one-dimensional character $\chi: G \to R^{\times}$, we denote by $R_{\chi}$ the $R[G]$-module that is uniquely defined by the scalar multiplication $rg \cdot s = r\chi(g)s$.

We can now show that the monomial representations of indecomposable cycle sets are induced:
\begin{thm} \label{thm:generic_monomials_are_induced}
Let $X$ be an indecomposable cycle set and $x_0$ be an element of $X$. We have the following isomorphisms:

\begin{enumerate}[label=\alph*)]
\item $\C(q)^X \cong \Ind_{G_0}^G \C(q)_{\chi_0}$, where $\C(q)^X$ is the $\C(q)[G]$-module associated with the monomial representation $\Theta$.
\item $\C[q^{\pm 1}]^X \cong \Ind_{G_0}^G \C[q^{\pm 1}]_{\chi_0}$, where $\C[q^{\pm 1}]^X$ is the $\C[q^{\pm 1}][G]$-module associated with the monomial representation $\Theta$.
\item $\C^X \cong \Ind_{\overline G_{l,0}}^{\overline G_l} \C_{\overline \chi_{l,0}}$, where $\C^X$ is the $\C[\overline G_l]$-module associated with the monomial representation $\overline\Theta_l$.
\end{enumerate}
\end{thm}

\begin{proof}
    We only deal with case b) as all other cases follow from a suitable extension/specialization of scalars. Therefore, writing $R = \C[q^{\pm 1}]$, it is easily seen that for $x \in X$, $g \in G$, we have $g\cdot Re_x = Re_{\lambda_g(x)}$, therefore $G$ permutes the family of submodules $(Re_x)_{x \in X}$. As $X$ is indecomposable, $G$ acts transitively on $X$, therefore $(Re_x)_{x \in X}$ is a system of imprimitivity for the $R[G]$-module $R^X$. Pick an $x_0 \in X$ and observe that
    \[
    G_0 = \{g \in G : \lambda_g(x_0) = x_0 \} = \{ g \in G: g\cdot(Re_{x_0}) = Re_{x_0} \},
    \]
    so \cref{pro:imprimitive_modules_are_induced} implies that there is an isomorphism of $R[G]$-modules
    \[
    R^X \cong \Ind_{G_0}^G Re_{x_0}.
    \]
    We are left with determining the character associated with the $R[G_0]$-module $Rx_0$. Let $g \in G_0$ and write $g = \sum_{x \in X}g_xx$, then by \cref{eq:theta_in_terms_of_g},
    \[
    g \cdot e_{x_0} = D_gP_ge_{x_0} = D_ge_{x_0} = q^{g_{x_0}}e_{x_0} = \chi_0(g)e_{x_0}.
    \]
    This proves that $Re_{x_0} \cong A_{\chi_0}$, as $R[G_0]$-modules. Therefore,
    \[
    R^X \cong \Ind_{G_0}^G R_{\chi_0}.\qedhere
    \]
\end{proof}

\section*{Acknowledgements}
This work is partially supported by the project OZR3762 of Vrije Universiteit Brussel, 
by FWO and CNRS via the International Emerging Actions project 328226  and
by Fonds Wetenschappelijk Onderzoek - Vlaanderen, via the Senior Research Project G004124N.

The first author expresses his gratitude to the Humboldt foundation that had supported him by means of a Feodor Lynen fellowship when beginning the research for this project.

The third author is supported by Fonds Wetenschappelijk Onderzoek - Vlaanderen, via a PhD Fellowship fundamental research, grant 11PIO24N. 

\emergencystretch=1em
\printbibliography
\end{document}